\documentclass{amsart}
\usepackage{latexsym}
\usepackage{amssymb}
\usepackage{epsfig}
\usepackage{amsthm}
\usepackage{latexsym}
\usepackage{longtable}
\usepackage{epsfig}
\usepackage{amsmath}
\usepackage{hhline}
\usepackage{color}
\newtheorem{theorem}{Theorem}[section]

\newtheorem{lm}[theorem]{Lemma}

\newtheorem{corollary}[theorem]{Corollary}

\newtheorem{pr}[theorem]{Proposition}

\newtheorem{ex}[theorem]{Example}

\DeclareMathOperator{\ch}{char}
\DeclareMathOperator{\GL}{GL}
\DeclareMathOperator{\SL}{SL}
\DeclareMathOperator{\Ber}{Ber}
\DeclareMathOperator{\Sym}{Sym}
\DeclareMathOperator{\Hom}{Hom}
\DeclareMathOperator{\Ext}{Ext}
\DeclareMathOperator{\Lie}{Lie}
\DeclareMathOperator{\trace}{trace}

\begin{document}
\title[Super-representations and  polynomial semi-invariants]{Super-representations of quivers and related polynomial semi-invariants}

\author{V.A.~Bovdi}
\address{Department of Mathematical Sciences, UAEU, Al-Ain, United Arab Emirates}
\email{vbovdi@gmail.com}

\author{A.N.~Zubkov}
\address{Sobolev Institute of Mathematics, Omsk Branch, Pevtzova 13, 644043 Omsk, Russia; Omsk State Technical University, Mira 11, 644050 Omsk, Russia}
\email{a.zubkov@yahoo.com}

\thanks{The work was supported by the UAEU UPAR grant G00002160. The second author was also supported by the RFFI grant 18-07-00526.}

\begin{abstract}
We introduce the notion of a super-representation of a quiver. For super-representations of quivers over a field of characteristic zero, we describe the corresponding (super)algebras of polynomial semi-invariants {and polynomial invariants}.
\end{abstract}
\maketitle

\section*{Introduction}

The notion of quiver representations has been introduced in \cite{gab} to provide a framework for a
wide range of linear algebra problems. Two representations of a quiver
are equivalent if they lie in the same orbit under a natural action of the product
of some general linear groups. Therefore, the problem of classifying representations of quivers
is the same as that of describing the orbit spaces of such actions. To approximate these orbit spaces by algebraic varieties, one has to determine polynomial invariants. For arbitrary quivers,
generators and relations of the algebra of polynomial invariants are described in \cite{brprocesi, dom1} in the case of the characteristic zero, and in \cite{don, zub} for the arbitrary characteristic.

Unfortunately, polynomial invariants do not distinguish orbits whose closures intersect nontrivially. Therefore we need to consider polynomial semi-invariants of quiver representations (see also \cite{ki, ka} for an additional motivation). The generators of the algebra of polynomial semi-invariants are described in \cite{domzub, derkweym, schofberg}.

The advent of the ideas of superization and quantization of classical mathematics brought the problem of developing super and quantum analogs of the above theory.

In the current article, we generalize the notion of quiver representations to super-representations of a quiver. More precisely, as in the classical (purely even) setting, a super-representation of a quiver is a collection of super-vector spaces, each of which is attached to a unique vertex of the quiver, and linear (not necessary parity-preserving) maps between them, each of which is assigned to a unique edge of the quiver.

The super-representations of a quiver, considered with parity-preserving (graded) morphisms, form an abelian category. Moreover, each full subcategory of super-representations of a fixed super-dimension can be regarded as an affine superscheme, on which a product of general linear supergroups acts naturally. Combining the main result of \cite{sergeev} with \cite{domzub, derkweym}, we calculate the polynomial semi-invariants {and polynomial invariants} of this action over a ground field of characteristic zero.

It turns out that the algebra of semi-invariants of super-representation of a quiver is rather small. If the quiver is acyclic and there are no extremal vertices for a super-dimension vector, then the algebra of semi-invariants of the corresponding super-representation is trivial. As a rule, such a strange phenomenon does not take place for any representation of quiver except for the case when at most one coordinate of its dimension vector is not zero.

Therefore, in the case of super-representations of quivers, the notion of a rational semi-invariant is more natural than the notion of a polynomial one. Indeed,  even if a quiver does not have extremal vertices with respect to a given super-dimension vector, one can construct a non-scalar (rational) semi-invariant as a Berezinian of some block matrix (see Example \ref{phenomena} later).

The principal problem is that the superalgebra of rational "functions" on the superscheme of super-representations of a quiver has a more complicated structure as a supermodule over
the corresponding product of general linear supergroups. In the forthcoming article, we are going to investigate rational semi-invariants in the simplest case of the Kronecker quiver.

\section{Affine supergroups and their representations}

Let $K$ be a field of characteristic $\ch K=p\neq 2$.

\subsection{Superspaces}

The prefix ``super" is a synonym of ``graded by the group $\mathbb{Z}_2=\{ 0 , 1 \}$".
For example, \emph{super-vector spaces} (or {\it superspaces}) are precisely $\mathbb{Z}_2$-graded vector spaces.
A superspace $V$ decomposes as $V = V_0\oplus V_1$, where the elements of $V_0$ are called \emph{even}, and elements of $V_1$ are called \emph{odd}. {We define a {\it parity function} $(V_0\sqcup V_1)\setminus\{0\}\to\{0, 1\}$ by $v\mapsto |v|=i$, provided $v\in V_i\setminus\{0\}$, $i\in\{0, 1\}$.}

The ordered pair $\dim V_0 |\dim V_1$ of non-negative integers is called the \emph{super-dimension} of $V$ and will be denoted by $s\dim V$.

The superspaces, together with parity-preserving (graded) morphisms, form the symmetric tensor category $\mathsf{SSpace}_K$. More precisely, if $V$ and $W$ are superspaces, then
the superspace structure on $V\otimes W$ is given as
\[
(V\otimes W)_{0}=V_0\otimes W_0 + V_1\otimes W_1, \qquad (V\otimes W)_{1}=V_0\otimes W_1 + V_1\otimes W_0.
\]
For a pair of super-vector spaces $V$ and $W$, the space of (ungraded) linear  maps $\Hom_K(V, W)$ has a natural structure of a super-vector space
given by
\[
\Hom_K(V, W)_i=\{\phi \;\mid\; \phi(V_j)\subseteq W_{i+j\pmod 2}\},
\]
where $i, j=0, 1$.
It follows that 
\[
\Hom_K(V, W)\simeq V^*\otimes W\quad \text{and}\quad  \Hom_K(V, W)_0 =\Hom_{\mathsf{SSpace}_K}(V, W).
\]

If we fix bases in the superspaces $V$ and $W$, then $\Hom_K(V, W)$ is isomorphic to the superspace of $(s+d)\times (s'+d')$ matrices $M_{s|d\times s'|d'}(K)$ with entries from $K$ such that
\[
(M_{s|d\times s'|d'}(K))_0=\left(\begin{array}{cc}
M_{s\times s'}(K) & 0 \\
0 & M_{d\times d'}(K)\end{array}\right)
\]
and
\[
(M_{s|d\times s'|d'}(K))_1=\left(\begin{array}{cc}
0 & M_{s\times d'}(K) \\
M_{d\times s'}(K) & 0\end{array}\right),
\]
where $s\dim V=s|d, s\dim W=s'|d'$.

\subsection{Parity {shift} functor}

The category $\mathsf{SSpace}_K$ admits a {\it parity shift} functor $V\mapsto\Pi V$, where
$(\Pi V)_i=V_{i+1\pmod 2}$ and  $i\in\{0, 1\}$. The functor $\Pi$ is a self-equivalence of the category $\mathsf{SSpace}_K$.
If $V$ is a superspace, then a homogeneous vector $v\in V$, regarded as a vector of $\Pi V$, is also denoted by $\Pi v$.

\subsection{Superalgebras and supermodules}

Let $A$ be an associative superalgebra, i.e. $A$ is a $\mathbb{Z}_2$-graded associative algebra.
A $\mathbb{Z}_2$-graded (left or right) $A$-module is called $A$-supermodule. The category of left (respectively, right) $A$-supermodules, with graded morphisms, is denoted by $A-\mathsf{Smod}$ (respectively, $\mathsf{Smod}-A$).

\subsection{Super-commutative superalgebras}

An associative superalgebra $A$ is cal\-led \emph{super-commutative} if it satisfies the identity
$ab=(-1)^{|a||b|}ba$ for all homogeneous elements $a, b\in A$. The typical example of a super-commutative superalgebra is a {\it super-symmetric} superalgebra $\Sym(V)$ of a superspace $V$. More precisely, $\Sym(V)$ is the quotient of the tensor (super)algebra $T(V)=\oplus_{l\geq 0} V^{\otimes l}$ modulo the (super)ideal generated by the elements
\[
v_1\otimes\cdots\otimes v_i\otimes v_{i+1}\otimes\cdots\otimes v_l-(-1)^{|v_i||v_{i+1}|}v_1\otimes\cdots\otimes v_{i+1}\otimes v_i\otimes\cdots\otimes v_l,
\]
for all {homogeneous vectors} $v_1, \ldots, v_l\in V$, where $1\leq i\leq l-1$ and $l\geq 2$. If $s\dim V = m|n$, then $\Sym(V)$ will be also denoted by $\Sym(m|n)$.

Let $\mathsf{SAlg}_K$ denote the category of super-commutative (unital) $K$-superalgebras with graded morphisms.

\subsection{Affine superschemes}

An {\it affine superscheme}  is a representable functor from the category $\mathsf{SAlg}_K$ to the category of sets $\mathsf{Set}$. Such functor,
say $\mathbb{X}$, is uniquely represented by a super-commutative superalgebra $K[\mathbb{X}]$, where 
\[
\mathbb{X}(A)=\Hom_{\mathsf{SAlg}_K}(K[\mathbb{X}], A), \qquad \forall A\in\mathsf{SAlg}_K.
\]
For example, an {\it affine superspace} $\mathbb{A}^{m|n}$ is represented by a superalgebra $\Sym(V)$, where $s\dim V=m|n$ and $\mathbb{A}^{m|n}(A)=A_0^m\oplus A_1^n$ for $A\in\mathsf{SAlg}_K$.

The superalgebra $K[\mathbb{X}]$ is naturally identified with the set of functor morphisms $Mor(\mathbb{X}, \mathbb{A}^{1|1})$ via
$f(x)=x(f)$, where $f\in K[\mathbb{X}], x\in\mathbb{X}(A)$ and $A\in\mathsf{SAlg}_K$ (cf. \cite{jan}, I.1.3(3)). In other words, the elements from $K[\mathbb{X}]$ are regarded as "functions" \ on the affine superscheme $\mathbb{X}$, and $K[\mathbb{X}]$ is called a {\it coordinate superalgebra} of $\mathbb{X}$. Additionally, the notion of a coordinate superalgebra $K[\mathbb{X}]$ makes sense for any functor $\mathbb{X}$ from $\mathsf{SAlg}_K$ to $\mathsf{Set}$.

Let $\mathbb{X}$ be an afine superscheme. We say that a subfunctor
$\mathbb{Z}\subseteq\mathbb{X}$ is a closed super-subscheme of $\mathbb{X}$, if there is an super-ideal $I$ of $K[\mathbb{X}]$ such that
\[\mathbb{Z}(A)=\{x\in\mathbb{X}(A)\mid x(I)=0\}, \qquad \forall A\in\mathsf{SAlg}_K.\]
\begin{ex}
\emph{Let $V$ be a superspace. Let $\mathbb{V}_a$ denote the functor \; $A\to V\otimes A$,\;  $A\in\mathsf{SAlg}_K$. It is obvious that $\mathbb{V}_a\simeq\mathbb{A}^{m+n|n+m}$, where $s\dim V=m|n$. Moreover, $V\to\mathbb{V}_a$ is a functor from the category $\mathsf{SSpace}_K$ to the category of affine superschemes. Its subfunctor $\mathbb{V}\subseteq\mathbb{V}_a$, such that $\mathbb{V}(A)=(V\otimes A)_0$ for any (super-commutative) superalgebra $A$, is a closed super-subscheme, which isomorphic to $\mathbb{A}^{m|n}$.}
\end{ex}
The functor $\mathbb{M}_{s|d\times s'|d'}$, associated with the superspace $M_{s|d\times s'|d'}(K)$,
is said to be  a \emph{superscheme of $s|d\times s'|d'$ matrices}. We have
\[
(M_{s|d\times s'|d'}(A))_0=\left(\begin{array}{cc}
M_{s\times s'}(A) & 0 \\
0 & M_{d\times d'}(A)\end{array}\right)
\]
and
\[
(M_{s|d\times s'|d'}(A))_1=\left(\begin{array}{cc}
0 & M_{s\times d'}(A) \\
M_{d\times s'}(A) & 0\end{array}\right)
\]
for any $A\in\mathsf{SAlg}_K$.

\subsection{Supergroups and supergroup actions}

An affine superscheme $\mathbb{G}$ is called an {\it affine supergroup} if $K[\mathbb{G}]$ is a Hopf superalgebra. In this case, $\mathbb{G}$ is a group functor, i.e., a functor from the category $\mathsf{SAlg}_K$ to the category of groups $\mathsf{Gr}$. An affine supergroup $\mathbb{G}$ is called \emph{algebraic}, provided $K[\mathbb{G}]$ is finitely generated. For a more detailed introduction to the theory of algebraic supergroups and Lie superalgebras, one can refer to \cite{ccf, kac, man, varadarajan, zub1}.

Let $\mathbb{G}$ be an affine supergroup. We say that $\mathbb{G}$ acts on an affine superscheme $\mathbb{X}$ on the left, if there is a superscheme morphism $\mathbb{G}\times\mathbb{X}\to\mathbb{X}$, such that for any $A\in\mathsf{SAlg}_K$ the map $\mathbb{G}(A)\times\mathbb{X}(A)\to\mathbb{X}(A)$
induces an action of the group $\mathbb{G}(A)$ on the set $\mathbb{X}(A)$. The definition of a right action is symmetric.

\subsection{Supermodules over supergroups}

Let $\mathbb G$ be an algebraic supergroup. The category of left rational $\mathbb G$-supermodules $\mathbb{G}-\mathsf{Smod}$, with graded morphisms,
coincides with the category of right $K[\mathbb{G}]$-supercomodules, with graded morphisms as well.

Equivalently, $V$ is a $\mathbb{G}$-supermodule if and only if $\mathbb{G}$ acts on $\mathbb{V}_a$ on the left, so that for any $A\in\mathsf{SAlg}_K$ the group $\mathbb{G}(A)$ acts on $V\otimes A$ by $A$-linear even automorphisms. Namely, if $V$ is a right $K[\mathbb{G}]$-supercomodule and $\tau_V : V\to V\otimes K[\mathbb{G}]$ is the corresponding comodule map, then every $g\in \mathbb{G}(A)$ acts on $V\otimes A$ by
\[
g(v\otimes a)=\big((\mathsf{id}_V\otimes g)\tau_V(v)\big)a,\qquad  \text{ where } v\in V \text{ and } a\in A.
\]
Note that the super-subscheme $\mathbb{V}$ is $\mathbb{G}$-invariant.

For any $\mathbb{G}$-supermodules $V$ and $W$, the tensor product $V\otimes W$ is a $\mathbb{G}$-supermodule via the diagonal action.
Thus the category $\mathbb{G}-\mathsf{Smod}$ is a symmetric tensor category. Moreover, the functor $\Pi$ induces a self-equivalence on $\mathbb{G}-\mathsf{Smod}$.

Let $V$ be a finite-dimensional $\mathbb G$-supermodule. Then both tensor and symmetric superalgebras $T(V)$ and $\Sym(V)$ are $\mathbb G$-supermodules considered via the diagonal action of $\mathbb{G}$. In particular, the elements of Lie superalgebra $\mathfrak{g}=\Lie(\mathbb{G})$ of $\mathbb{G}$ act on $V^{\otimes l}$ as superderivations, i.e.,
\[
\phi(v_1\otimes\cdots\otimes v_l)=\sum_{1\leq k\leq l}(-1)^{|\phi|(\sum_{1\leq j< i}|v_j|)}v_1\otimes\cdots\otimes\phi(v_i)\otimes\cdots\otimes v_l
\]
for every $\phi\in\mathfrak{g}$.
Elements of $\mathfrak{g}$ act on $\Sym^l(V)$ as
\[
\phi(v_1\cdots v_l)=\sum_{1\leq i\leq l}(-1)^{|\phi|(\sum_{1\leq j< i}|v_j|)}v_1\cdots\phi(v_i)\cdots v_l.
\]
The dual superspace $V^*$ has a uniquely defined structure of a right $K[\mathbb{G}]$-super\-co\-mo\-dule given by $\tau_{V^*}(f)=\sum f_{(1)}\otimes f_{(2)}$, where
\[\sum (-1)^{|f_{(2)}||v_{(1)}|}f_{(1)}(v_{(1)})f_{(2)}v_{(2)}=f(v), \forall v\in V \ \mbox{and} \ \tau_V(v)=\sum v_{(1)}\otimes v_{(2)}.\]

\subsection{General linear and special linear supergroups}

Let $K[x_{ij} \mid 1\leq i, j\leq m+n]$ be a polynomial superalgebra, freely generated by $(m+n)^2$  variables $x_{ij}$, where
$|x_{ij}|=0$ if $1\leq i, j\leq m$ or $m+1\leq i, j\leq m+n$, and $|x_{ij}|=1$ otherwise. In other words, $K[x_{ij} \mid 1\leq i, j\leq m+n]\simeq \Sym(W)$, where $W$ is a superspace with basis $x_{ij}, 1\leq i, j\leq m+n$.

The {\it generic} matrix $X=X(m|n)$ has a block form
\[
X=\left(\begin{array}{cc}
X_{00} & X_{01} \\
X_{10} & X_{11}\end{array}\right),
\]
where
\[
X_{00}=(x_{ij})_{1\leq i, j\leq m}, \qquad\qquad X_{01}=(x_{ij})_{1\leq i\leq m < j\leq m+n},
\]
\[
X_{10}=(x_{ij})_{1\leq j\leq m< i\leq m+n}, \qquad X_{11}=(x_{ij})_{m+1\leq i, j\leq m+n}.
\]
Denote by $str(X)=\trace(X_{00})-\trace(X_{11})$ the supertrace of $X$.

The general linear supergroup $\GL(m|n)$ is represented by the Hopf superalgebra
\[K[\GL(m|n)]=K[x_{ij}\mid 1\leq i, j\leq m+n]_{\det(X_{00})\det(X_{11})},\]
{the localization of $K[x_{ij}\mid 1\leq i, j\leq m+n]$ with respect to $\det(X_{00})\det(X_{11})$.} The super-bialgebra structure on $K[\GL(m|n)]$ is given by
\[
\Delta(x_{ij})=\sum_{1\leq k\leq m+n}x_{ik}\otimes x_{kj},\qquad  \epsilon(x_{ij})=\delta_{ij}, \quad \text{ where } 1\leq i, j\leq m+n,
\]
and the antipode $s$ defined as follows. Denote $y_{ij}=s(x_{ij})$. If $Y$ is the matrix
$(y_{ij})_{1\leq i, j\leq m+n}$,  then
\[Y=\left(\begin{array}{cc}
(X_{00} -X_{01}X_{11}^{-1}X_{10})^{-1} & -X_{00}^{-1}X_{01}(X_{11}-X_{10}X_{00}^{-1}X_{01})^{-1} \\
-X_{11}^{-1}X_{10}(X_{00} -X_{01}X_{11}^{-1}X_{10})^{-1} & (X_{11}-X_{10}X_{00}^{-1}X_{01})^{-1} \end{array}\right).\]

The {\it Berezinian} $\Ber_{m|n}(X)$ is a rational function given by
\[
\Ber_{m|n}(X)=\det(X_{00}-X_{01}X_{11}^{-1}X_{10})\det(X_{11})^{-1}.
\]
The element $\Ber_{m|n}(X)$ is group-like, that is
\[\Delta(\Ber_{m|n}(X))=\Ber_{m|n}(X)\otimes \Ber_{m|n}(X).\]
Equivalently, the map $g\mapsto \Ber_{m|n}(g)$, where $g\in \GL(m|n)(A)$ and $A\in\mathsf{SAlg}_K,$ is a character of the supergroup $\GL(m|n)$.

If $n=0$ ($m=0$, respectively), then $\Ber(X)=\det(X)$ \; \big($\Ber(X)=\det(X)^{-1}$, respectively\big).
Additionally, the coordinate (super)algebra $K[\GL(m|n)]$ is isomorphic to the localization $K[x_{ij} \mid 1\leq i, j\leq m+n]_{\Ber_{m|n}(X)}$.

As a group functor, $\GL(m|n)$ can also be defined as follows. Let $V$ be a super-vector space of super-dimension $m|n$. For any super-commutative superalgebra $A$, the group $\GL(m|n)(A)$ consists of all even $A$-linear automorphisms of super-vector space $V\otimes A$. In particular, $V$ has a {\it canonical} $\GL(V)$-supermodule structure given by
\[
v_i\mapsto\sum_{1\leq j\leq m+n} v_j\otimes x_{ji}, \quad \text{ where }\quad  1\leq i\leq m+n,
\]
with respect to a homogeneous basis $v_1, \ldots, v_{m+n}$ such that $|v_i|=0$ for $1\leq i\leq m$, and $|v_i|=1$ otherwise. Using this interpretation, we will denote $\GL(m|n)$ by $\GL(V)$.

Observe also that there is a one-to-one correspondence between homomorphisms of algebraic supergroups
$\mathbb{G}\to \GL(V)$ and $\mathbb{G}$-supermodule structures on $V$.

The special linear supergroup $\SL(V)=\SL(m|n)$ is a closed super-subgroup of $\GL(m|n)$ given by
$\SL(m|n)(A)=\{g\in \GL(m|n)(A) \mid \Ber_{m|n}(g)=1\}$ for any superalgebra $A$.

\subsection{Symmetric group}

For a non-negative integer $n$, let $[1, n]$ denote the interval $\{1,\ldots, n\}$. Let $S_n$ denote the symmetric group, consisting of all permutations of $[1, n]$.

\section{Super-representations of a quiver}

\subsection{Standard setting}

Let $Q=(Q_0, Q_1, t, h)$ be a quiver with the vertex set $Q_0$, the edge set $Q_1$, and functions $t, h : Q_1 \to Q_0$ assigning to an edge $e\in Q_1$ its {\it tail} $t(e)$ and {\it head} $h(e)$. We assume that all quivers are finite, i.e., their vertex sets and edge sets are finite.

A {\it path} $p$ in $Q$ is an ordered sequence of edges, say $e_1, \ldots, e_k$, such that $h(e_{i+1})=t(e_i)$, $1\leq i\leq k-1$. Additionally, if $h(e_1)=t(e_k)$, then $p$ is called a {\it closed path}.

A {\it super-representation} $V$ of $Q$ consists of super-vector spaces $V(a)$ for $a\in Q_0$ and (ungraded) linear morphisms $V(e)\in \Hom_K(V(t(e)), V(h(e)))$ for $e\in Q_1$. If $V$ and $W$ are two super-representations of $Q$, then a morphism $\phi : V\to W$ is a collection of superspace morphisms $\phi(a) : V(a)\to W(a)$ for $a\in Q_0$ such that for every $e\in Q_1$ there is $W(e)\phi(t(e))=\phi(h(e))V(e)$.

Let $\widehat{Q}$ denote an \emph{edge doubled} quiver with 
\[
\widehat{Q}_0=Q_0, \quad \widehat{Q}_1=\{e_0\mid e\in Q_1\}\sqcup \{e_1\mid e\in Q_1\}.
\] 
Moreover, for any $e\in Q_1$ there is $h(e_i)=h(e_i)$, $t(e_i)=t(e_i)$, $i\in\{0, 1\}$. Then the \emph{path algebra} $K\widehat{Q}$ has the unique superalgebra structure, such that $|e_i|=i$, $e\in Q_1$, $i\in\{0, 1\}$, and for any $a\in Q_0$ the idempotent $e_a$ is even.
\begin{lm}\label{representations and supermodules}
The abelian category $\mathsf{SRep}(Q)$ of super-representations of the quiver $Q$ is equivalent to
the category of $K\widehat{Q}-\mathsf{Smod}$.
\end{lm}
\begin{proof}
The functor $K\widehat{Q}-\mathsf{Smod}\to\mathsf{SRep}(Q)$ associates with a  $K\widehat{Q}$-supermodule $V$ a super-representation, denoted by the same symbol $V$, such that $V(a)=e_a V, a\in Q_0,$ and $V(e)=e_0|_{e_{t(e)}V}\oplus e_1|_{e_{t(e)}V}$ for any $e\in Q_1$. We leave for the reader the routine checking that this is an equivalence of categories.
\end{proof}
For any $V\in K\widehat{Q}-\mathsf{Smod}$ we have a standard exact sequence
\[
\begin{split}
0\to\oplus_{f\in\widehat{Q}_1}K\widehat{Q}e_{h(f)}\otimes e_{t(f)}\Pi^{|f|}V&\to\oplus_{a\in\widehat{Q}_0}K\widehat{Q}e_a\otimes e_a V\to V\to 0,\\
r\otimes \Pi^{|f|}v&\mapsto rf\otimes v-r\otimes fv, \quad  \\
s\otimes w\mapsto sw, \quad r\in K\widehat{Q}e_{h(f)}, \quad &s\in K\widehat{Q}e_a, \quad v\in e_{t(f)}V,\quad  w\in e_a V,
\end{split}
\]
which is obviously a projective resolution in $K\widehat{Q}-\mathsf{Smod}$, hence in $\mathsf{SRep}(Q)$ as well (see \cite{craw, derkweym, schof} for more details). In particular, the global dimension of $\mathsf{SRep}(Q)$ is at most $1$, that is for all $V, W\in\mathsf{SRep}(Q)$ and $i\geq 2$, we have $\Ext_{\mathsf{SRep}(Q)}^i(V, W)=0$.

Analogously to the classical setting, let $s\underline{\dim} V$ denote the {\it super-dimension vector} of $V$ defined by
$(s\underline{\dim}V)(a)=s\dim V(a)$ for $a\in Q_0$. If $\alpha$ and $\beta$ are two super-dimension vectors, then the {\it Ringel} super-form $\langle \alpha, \beta\rangle$ is defined as
\[
\begin{split}
\sum_{a\in Q_0}\big(\alpha(a)_0\beta(a)_0&+\alpha(a)_1\beta(a)_1\big)\\
&-\sum_{e\in Q_1}\big(\alpha(t(e))_0 + \alpha(t(e))_1)(\alpha(h(e))_0+\alpha(h(e))_1).
\end{split}
\]
\begin{pr}\label{Ringel formula}
Let $V$ and $W$ are finite dimensional super-representations of $Q$. Then we have
\[\langle s\underline{\dim} V, s\underline{\dim} W\rangle=\dim \Hom_{\mathsf{SRep}(Q)}(V, W)-\dim \Ext_{\mathsf{SRep}(Q)}(V, W).
\]
\end{pr}
\begin{proof}
Using the above projective resolution, one obtains
\[0\to\Hom_{\mathsf{SRep}(Q)}(V, W)\to\Hom_{\mathsf{SRep}(Q)}(\oplus_{a\in\widehat{Q}_0}K\widehat{Q}e_a\otimes e_a V, W)\to\]
\[\Hom_{\mathsf{SRep}(Q)}(\oplus_{f\in\widehat{Q}_1}K\widehat{Q}e_{h(f)}\otimes e_{t(f)}\Pi^{|f|}V, W)\to\Ext^1_{\mathsf{SRep}(Q)}(V, W)\to 0.\]
Since
\[\Hom_{\mathsf{SRep}(Q)}(\oplus_{a\in\widehat{Q}_0}K\widehat{Q}e_a\otimes e_a V, W)\simeq
\oplus_{a\in\widehat{Q}_0}\Hom_{\mathsf{SSpace}_K}(e_aV, e_a W)\]
and
\[
\begin{split}
\Hom_{\mathsf{SRep}(Q)}(\oplus_{f\in\widehat{Q}_1}K\widehat{Q}e_{h(f)}&\otimes e_{t(f)}\Pi^{|f|}V, W)\\
&\simeq
\oplus_{f\in\widehat{Q}_1}\Hom_{\mathsf{SSpace}_K}(e_{h(f)}\Pi^{|f|}V, e_{t(f)}W),
\end{split}
\]
the statement obviously follows.
\end{proof}

\subsection{Another equivalence}
The abelian category $\mathsf{SRep}(Q)$ of super-rep\-re\-sen\-ta\-tions of the quiver $Q$ is equivalent to the category of representations of an \emph{edge and vertex doubled}  quiver $\widetilde{Q}$ constructed as follows. Set $\widetilde{Q}_0=\{a_i \mid a\in Q_0, i=0,1\}$, $\widetilde{Q}_1=\{e_{ij}\mid e\in Q_1, i,j=0, 1\}$, $h(e_{ij})=h(e)_j$ and $t(e_{ij})=t(e)_i$. The equivalence $\mathsf{SRep}(Q)\simeq\mathsf{Rep}(\widetilde{Q})$ is given by $V\mapsto\widetilde{V}$, where $\widetilde{V}(a_i)=V(a)_i$ and $\widetilde{V}(e_{ij})=V(e)_{ij}$ for $a\in Q_0, e\in Q_1$ and $i, j=0, 1$. Since
$\langle s\underline{\dim} V, s\underline{\dim} W\rangle=\langle \underline{\dim}\widetilde{V}, \underline{\dim}\widetilde{W}\rangle$, all the above statements follow.

\subsection{Another parity shift functor}
For a {\it parity vector} $b\in\mathbb{Z}_2^{Q_0}$, one can define a self-equivalence \quad $\Pi^b : \mathsf{SRep}(Q)\to\mathsf{SRep}(Q)$ via $V\mapsto\Pi^b V,$ where
$(\Pi^b V)(a)=\Pi^{b(a)}V(a)$ and $(\Pi^b V)(e)=V(e)$ for every $a\in Q_0$ and $e\in Q_1$.

If $s\underline{\dim} V=\alpha$, then $\Pi^b V$ has a super-dimension $\Pi^b\alpha$ such that for every $a\in Q_0$ there is $(\Pi^b\alpha)(a)_i=\alpha(a)_{i+b(a)\pmod 2}$ for
$i\in\{0, 1\}$. Furthermore, for any super-dimension vectors $\alpha$ and $\beta$, there holds $\langle\alpha, \beta\rangle=\langle\Pi^b\alpha, \Pi^b\beta\rangle$.

\subsection{Affine superschemes associated with super-representations of quivers}

For each $a\in Q_0$, fix a basis of $V(a)$, say $v_i(a), 1\leq i\leq \alpha(a)_0+\alpha(a)_1$, such that $|v_i(a)|=0$ if $1\leq i\leq\alpha(a)_0$, and $|v_i(a)|=1$ otherwise.
The dual basis of $V(a)^*$ consists of the vectors $v_i^*(a)$ such that $v_i^*(a)(v_j(a))=\delta_{ij}$ for $1\leq i, j\leq \alpha(a)_0+\alpha(a)_1$. Then $\mathsf{SRep}_{\alpha}(Q)$ can be identified with a superspace
\[\oplus_{e\in Q_1}M_{\alpha(h(e))\times\alpha(t(e))}(K).\]
The supergroup $\GL(\alpha)=\prod_{a\in Q_0}\GL(\alpha(a))$ acts on the associated affine superspace
\[\prod_{e\in Q_1}\mathbb{M}_{\alpha(h(e))\times\alpha(t(e))}\]
by
\[(\prod_{a\in Q_0}g(a))\cdot(\prod_{e\in Q_1}M(e))=\oplus_{e\in Q_1}g(h(e))M(e)g(t(e))^{-1},\] \[\prod_{a\in Q_0}g(a)\in GL(\alpha)(A), \prod_{e\in Q_1}M(e)\in \prod_{e\in Q_1}\mathbb{M}_{\alpha(h(e))\times\alpha(t(e))}(A), A\in\mathsf{SAlg}_K.\]
We denote the superscheme $\prod_{e\in Q_1}\mathbb{M}_{\alpha(h(e))\times\alpha(t(e))}$ by
the same symbol $\mathsf{SRep}_{\alpha}(Q)$, if it does not lead to confusion.

The coordinate superalgebra $K[\mathsf{SRep}_{\alpha}(Q)]$ is isomorphic to the polynomial superalgebra freely generated by the elements
$x_{ij}(e)$,  where
\[
 e\in Q_1, \qquad 1\leq i\leq \alpha(h(e))_0+\alpha(h(e))_1, \qquad  1\leq j\leq \alpha(t(e))_0+\alpha(t(e))_1.
\]
Here $|x_{ij}(e)|=0$ provided $1\leq i\leq \alpha(h(e))_0$ and $1\leq j\leq \alpha(t(e))_0$, or $\alpha(h(e))_0+1\leq i\leq \alpha(h(e))_0+\alpha(h(e))_1$ and $\alpha(t(e))_0+1\leq j\leq \alpha(t(e))_0+\alpha(t(e))_1$; and $|x_{ij}(e)|=1$ otherwise. In what follows, let $X_{\alpha}(e)$ denote the generic matrix
\[
\big(x_{ij}(e)\big)_{1\leq i\leq \alpha(h(e))_0+\alpha(h(e))_1, 1\leq j\leq \alpha(t(e))_0+\alpha(t(e))_1}.
\]

Since $\GL(\alpha)$ acts naturally on the affine superscheme $\mathsf{SRep}_{\alpha}(Q)$, hence this supergroup acts on its coordinate superalgebra as well. Moreover, $K[\mathsf{SRep}_{\alpha}(Q)]$ is isomorphic to
\[\Sym((\oplus_{e\in Q_1}\Hom_K(V(t(e)), V(h(e)))^*)\simeq \Sym(\oplus_{e\in Q_1}V(h(e))^*\otimes V(t(e)))\] as a $\GL(\alpha)$-supermodule. This isomorphism is induced by the map
\[
x_{ij}(e)\mapsto v_i^*\big(h(e)\big)\otimes v_j\big(t(e)\big), \text{ where }  e\in Q_1.
\]
Furthermore, for any parity vector $b$,
the supergroup $\GL(\alpha)$ also acts on the associated affine superscheme $\mathsf{SRep}_{\Pi^b\alpha}(Q)$, so that the superalgebra isomorphism
\[
K[\mathsf{SRep}_{\Pi^b\alpha}(Q)]\simeq \Sym(\oplus_{e\in Q_1}\Pi^{b(h(e))}V(h(e))^*\otimes \Pi^{b(t(e))}V(t(e)))
\]
is an isomorphism of $\GL(\alpha)$-supermodules.
Here, $K[\mathsf{SRep}_{\Pi^b\alpha}(Q)]$ is freely generated by the elements $\Pi^b x_{ij}(e)$ such that
\[
|\Pi^b x_{ij}(e)|=|x_{ij}(e)|+b(h(e))+b(t(e))\pmod 2,
\]
and the above isomorphism is given by
\[
\Pi^b x_{ij}(e)\mapsto \Pi^{b(h(e))}v_i^*(h(e))\otimes\Pi^{b(t(e))}v_j(t(e)).
\]

\section{Polynomial semi-invariants of super-representations of quivers}

Despite the fact that the super-representation theory of $Q$ is a copy of the representation theory of its companion $\widetilde{Q}$, the semi-invariants of super-representations of $Q$ (of a fixed super-dimension) with respect to the action of the corresponding product of general linear supergroups are not related to the semi-invariants of representations of $\widetilde{Q}$ (of the corresponding dimension) with respect to the action of the relevant product of general linear groups.

Moreover, even the definition of semi-invariants of super-representations of a quiver is not immediately apparent and needs extra clarification.

As it has been observed, $K[\mathsf{SRep}_{\Pi^b\alpha}(Q)]$ is a polynomial $\GL(\alpha)$-supermodule. We call a polynomial $f\in K[\mathsf{SRep}_{\Pi^b\alpha}(Q)]$ a {\it semi-invariant} of weight $\chi$, where $\chi$ is a character of the supergroup $\GL(\alpha)$ if
\[
g\cdot (f\otimes 1)=f\otimes \chi(g) \quad \text{ for every }  g\in \GL(\alpha)(A) \text{ and } A\in\mathsf{SAlg}_K.
\]
If $\ch K=0$, then the character group of $\GL(\alpha)$ is generated by the characters
\[
g=\prod_{b\in Q_0} g(b)\mapsto \Ber_{\alpha(a)}(g(a)), \text{ where } g\in \GL(\alpha)(A), A\in\mathsf{SAlg}_K,
\]
and $a$ runs over $Q_0$. Therefore, the super-subalgebra of semi-invariants coincides with $K[\mathsf{SRep}_{\Pi^b\alpha}(Q)]^{\SL(\alpha)}$, where $\SL(\alpha)=\prod_{a\in Q_0}\SL(\alpha(a))$.

On the other hand, if $\ch K=p>0$, then the character group of $\GL(\alpha)$ is generated by the characters
\[
g=\prod_{b\in Q_0}g(b)\mapsto \Ber_{\alpha(a)}(g(a))^u\det(g(a)_{11})^{pv}, 
\]
where  $g\in \GL(\alpha)(A)$, $A\in\mathsf{SAlg}_K$, $u, v=0, 1, u+v\leq 1$ and $a$ runs over $Q_0$ (see Lemma 13.5, \cite{zub2}).
Thus the superalgebra of semi-invariants contains \newline $K[\mathsf{SRep}_{\Pi^b\alpha}(Q)]^{\SL(\alpha)}$ as a super-subalgebra.
In what follows, we denote the superalgebra of semi-invariants by $\mathsf{SSI}_Q(\alpha, b)$ and its super-subalgebra $K[\mathsf{SRep}_{\Pi^b\alpha}(Q)]^{\SL(\alpha)}$ by $\mathsf{SSi}_Q(\alpha, b)$. As it has been already observed, $\mathsf{SSI}_Q(\alpha, b)=\mathsf{SSi}_Q(\alpha, b)$ provided $\ch K=0$.

The superalgebra of polynomial invariants $K[\mathsf{SRep}_{\Pi^b\alpha}(Q)]^{\GL(\alpha)}$
is denoted by $\mathsf{SI}_Q(\alpha, b)$.

The superalgebra $K[\mathsf{SRep}_{\Pi^b\alpha}(Q)]$ is $\mathbb{N}^{Q_1}$-graded in such a way that
each $\Pi^b x_{ij}(e)$ has a multi-degree
\[
\delta_e=(0, \ldots, \underbrace{1}_{e-\mbox{th place}}, \ldots , 0),
\]
where $e\in Q_1, 1\leq i\leq \alpha(t(e))_0+\alpha(t(e))_1$ and $1\leq j\leq \alpha(h(e))_0+\alpha(h(e))_1$. For each
${\bf n}=\sum_{e\in Q_1} n_e\delta_e\in\mathbb{N}^{Q_1}$, the corresponding
homogeneous component
\[
K[\mathsf{SRep}_{\Pi^b\alpha}(Q)]({\bf n})\simeq \otimes_{e\in Q_1}\Sym^{n_e}\big(\Pi^{b(h(e))}V(h(e))^*\otimes\Pi^{b(t(e))}V(t(e))\big)
\]
is a $\GL(\alpha)$-super-submodule of $K[\mathsf{SRep}_{\Pi^b\alpha}(Q)]$. Therefore, the superalgebras
\[
\mathsf{SSI}_Q(\alpha, b), \mathsf{SSi}_Q(\alpha, b)\quad \text{and}\quad \mathsf{SI}_Q(\alpha, b)
\] 
 have the induced $\mathbb{N}^{Q_1}$-gradings.

\section{A description of polynomial semi-invariants of super-representations of quivers}

From now on, $K$ is a field of characteristic zero.

\subsection{Polarization operators}

Let $Q$ be a quiver and $f\in \mathsf{SSI}_Q(\alpha, b)$ be its semi-in\-va\-riant of multi-degree ${\bf n}$. Let $Q(\bf n)$ denote a $\bf n$-{\it polarized} quiver such that $Q({\bf n})_0 =Q_0$ and $Q({\bf n})_1=\sqcup_{e\in Q_1}\{e_1, \ldots , e_{n_e}\}$, where $t(e_i)=t(e), h(e_i)=h(e)$ for $1\leq i\leq n_e$ and $e\in Q_1$. Then
\[
f|_{\Pi^b X_{\alpha}(e)\mapsto\sum_{1\leq i\leq n_e} t(e_i)\Pi^b X_{\alpha}(e_i), e\in Q_1}=
\sum_{{\bf s}}\big(\prod_{e\in Q_1}\prod_{1\leq i\leq n_e}t(e_i)^{s(e_i)}\big) f_{{\bf s}},
\]
where the sum runs over all ${\bf s}\in\mathbb{N}^{Q({\bf n})_1}$ such that $\sum_{1\leq i\leq n_e}
s(e_i) =n_e, e\in Q_1$, and $f_{{\bf s}}$ belongs to $(\mathsf{SSI}_{Q({\bf n})}(\alpha, b))({\bf s})$.
In other words, for every $\bf s$ as above, one can define a {\it polarization operator} $\mathsf{P}_{\bf s} : f\mapsto f_{\bf s}$ that maps $(\mathsf{SSI}_Q(\alpha, b))({\bf n})$ to $(\mathsf{SSI}_{Q({\bf n})}(\alpha, b))({\bf s})$.

Let $T$ be a subset of $Q_1$. The {\it partial linearization} $P_T$ of $f$ in $T$ is given by $\mathsf{P}_T f=f_{{\bf s}}$, where $s(e_i)=1$ for $1\leq i\leq n_e$ when
$e\in T$; and $s(e_1)=n_e$ and $s(e_i)=0$ for $2\leq i\leq n_e$ when $e\notin T$. As in the purely even case, there is the identity
\[
(\star) \qquad \qquad  (\mathsf{P}_T f)|_{\Pi^b X_{\alpha}(e_i)\mapsto\Pi^b X_{\alpha}(e), e\in T, 1\leq i\leq n_e}=\big(\prod_{e\in T} n_e !\big) f.
\]
The partial linearization of $f$ in $Q_1$ is called a {\it complete linearization} of $f$.

The identity $(\star)$ shows that a description of semi-invariants of a given multi-degree
can be derived from a description of multi-linear semi-invariants of the polarized quiver.

\subsection{Sink and source vertices}

For every $a\in Q_0$, let $\mathsf{IN}(a)$ and $\mathsf{OUT}(a)$ denote the sets $\{e\in Q_1\mid h(e)=a\}$ and $\{e\in Q_1\mid t(e)=a\}$, respectively. We denote their cardinalities by $\mathsf{in}(a)$ and $\mathsf{out}(a)$, correspondingly. A vertex $a\in Q_0$ is a {\it source} (a {\it sink}, respectively) if  $\mathsf{in}(a)=0$
($\mathsf{out}(a)=0$, respectively).

Let $\mathsf{SOURCE}(Q)$ and $\mathsf{SINK}(Q)$ denote the subsets of $Q_0$, consisting of sources and sinks respectively.

\subsection{Extremal and ordinary vertices}

For each $a\in Q_0$ denote $d_{\alpha}(a)=\alpha(a)_0+\alpha(a)_1$.
A vertex $a\in Q_0$ is {\it extremal} for the superdimension vector $\alpha$, provided
$\alpha(a)_0\alpha(a)_1=0$. Otherwise, $a$ is {\it ordinary}.

\subsection{Typical (semi)invariants}

If $p$ is a path in the quiver $Q$, consisting of the edges $e_1, \ldots, e_k$, then denote by $\Pi^b X_{\alpha}(p)$ the matrix
 $\Pi^b X_{\alpha}(e_k)\cdots\Pi^b X_{\alpha}(e_1)$. The subindex $\alpha$ is omitted if it does not lead to confusion.

If $p$ is a closed path, then the supertrace  $\mathrm{str}(\Pi^b X(p))$ belongs to $\mathsf{SI}_Q(\alpha, b)\subseteq\mathsf{SSI}_Q(\alpha, b)$.

Besides such supertraces, there are other semi-invariants which can be constructed in the same way as in \S 4 of \cite{domzub}.
We simplify the definitions from \cite{domzub}, since we are not going to calculate the basis of the superalgebra of semi-invariants as a vector space.

Let $A=\{a_1, \ldots, a_k\}$ and $C=\{c_1, \ldots, c_l\}$, respectively, be arbitrary subsets of $\mathsf{SINK}(Q)$ and $\mathsf{SOURCE}(Q)$, respectively. Assume additionally, that both $A$ and $C$ consist of extremal vertices and all superspaces $\Pi^{b(a_i)}V(a_i)$ and $\Pi^{b(c_j)}V(c_j)$ are purely even.

Let $P$ be a collection of paths in $Q$ such that the terminal and initial points of each path from $P$ belong to $A$ and $C$, respectively. Fix two collections of positive integers, say $\{q_1, \ldots, q_k\}$ and $\{r_1, \ldots, r_l\}$, such that
\[
t=\sum_{1\leq i\leq k}d_{\alpha}(a_i) q_i=\sum_{1\leq j\leq l}d_{\alpha}(c_j) r_j .
\]
We construct a $t\times t$ matrix $X$, partitioned into $(q_1+\cdots+ q_k)\times (r_1+\cdots+ r_l)$ blocks. The size of the block at the position $(a,b)$ is $d_{\alpha}(a_i)\times d_{\alpha}(c_j)$, where $a$ and $b$ satisfy the inequalities
\[
\begin{split}
q_1+\cdots +q_{i-1}+1\leq &a\leq q_1+\cdots+q_i,\\
r_1+\cdots +r_{j-1}+1\leq &b\leq r_1+\cdots+r_j .
\end{split}
\]
The block at the position $(a,b)$ is chosen as an arbitrary linear combination of matrices $\Pi^b X(p)$,
where $p$ runs over all paths from $P$ starting at $c_j$ and ending at $a_i$.

Since $\det(X)$ is a semi-invariant in $\mathsf{SSI}_Q(\alpha, b)$, every homogeneous component of it is again a semi-invariant in $\mathsf{SSI}_Q(\alpha, b)$. It is called a {\it determinant-like} semi-invariant of $Q$.

\subsection{Sink/source normalizing}

Let $a$ be a vertex of the quiver $Q$, that is neither source nor sink.
We construct the quiver $Q'$ as follows. The vertices of the quiver $Q'$ consist of vertices of $Q$ together with one additional vertex $a'$. Every edge of $Q$ that does not terminate in $a$ is also an edge in $Q'$. Every edge in $Q$ that ends in $a$ is replaced by the edge in $Q'$ starting in the same vertex and terminating in $a'$. Finally, we add an edge $e(a)$ that begins in $a'$ and ends in $a$.

Furthermore, if $\alpha$ is a super-dimension vector corresponding to $Q$, then  let $\alpha'$ denote the super-dimension vector corresponding to $Q'$ such that $\alpha'(c)=\alpha(c)$ for $c\in Q_0$ and $\alpha'(a')=\alpha(a)$. Finally,
define $b'\in\mathbb{Z}_2^{Q'_0}$ by $b'(c)=b(c)$ for every $c\in Q_0$, and $b'(a')=b(a)$.

The above procedure, which replaces the quiver $Q$ by the quiver $Q'$ and the super-dimension vector $\alpha$ by the super-dimension vector $\alpha'$, is called a \emph{sink/source normalizing at the vertex $a$}.

\subsection{Reduction to the case when any extremal vertex is either source or sink}

The multi-linear component of $\mathsf{SSI}_Q(\alpha, b)$ is isomorphic to
\[
(\otimes_{a\in Q_0}((\Pi^{b(a)}V(a)^*)^{\otimes\mathsf{in}(a)}\otimes \Pi^{b(a)}V(a)^{\otimes\mathsf{out}(a)}))^{\SL(\alpha)}.
\]
Since each supergroup $\SL(\alpha)$ is connected, Lemma 9.5 from \cite{zub1} implies
\[
\begin{split}
(\otimes_{a\in Q_0}((\Pi^{b(a)}V(a)^*)^{\otimes\mathsf{in}(a)}\otimes \Pi^{b(a)}V(a)^{\otimes\mathsf{out}(a)}))^{\SL(\alpha)}&\simeq\\
(\otimes_{a\in Q_0}((\Pi^{b(a)}V(a)^*)^{\otimes\mathsf{in}(a)}\otimes \Pi^{b(a)}V(a)^{\otimes\mathsf{out}(a)}))^{\mathfrak{sl}(\alpha)}&\simeq\\
\otimes_{a\in Q_0} ((\Pi^{b(a)}V(a)^*)^{\otimes\mathsf{in}(a)}\otimes \Pi^{b(a)}V(a)^{\otimes\mathsf{out}(a)})^{\mathfrak{sl}(\alpha(a))}.&
\end{split}
\]
In the notations from \cite{sergeev}, if $b(a)=0$, then the superspace
\[
\Big((\Pi^{b(a)}V(a)^*)^{\otimes\mathsf{in}(a)}\otimes\Pi^{b(a)}V(a)^{\otimes\mathsf{out}(a)}\Big)^{\mathfrak{sl}(\alpha(a))}
\]
can be identified with the multi-linear component of the superalgebra $\mathfrak{A}^{\mathsf{in}(a), 0}_{\mathsf{out}(a), 0}$. If $b(a)=1$, then it can be identified with the multi-linear component of the superalgebra $\mathfrak{A}^{0, \mathsf{in}(a)}_{0, \mathsf{out}(a)}$.
\begin{pr}\label{reduction}
Let $a$ be an extremal vertex of the quiver $Q$ (with respect to a super-dimension vector $\alpha$). Assume that $\alpha(a)_0=0$ (the case when $\alpha(a)_1=0$ is analogous). Let $Q'$ and $\alpha'$ be obtained by sink/source normalizing.
The superalgebra morphism $\mathsf{SSI}_{Q'}(\alpha', b')\to\mathsf{SSI}_Q(\alpha, b)$, induced by the identity map on $\Pi^{b'}X_{\alpha'}(e)=\Pi^b X_{\alpha}(e)$ for $e\in Q_1$, and which takes $\Pi^{b'} X_{\alpha'}(e(a))=X_{\alpha'}(e(a))$ to the identity matrix $E=E_{\alpha(a)_1\times\alpha(a)_1}$, is surjective.
\end{pr}
\begin{proof}
We need to prove that for any multi-linear semi-invariant $f\in \mathsf{SSI}_Q(\alpha, b)$ there is a semi-invariant $g\in \mathsf{SSI}_{Q'}(\alpha', b')$ such that our superalgebra morphism takes $g$ to $f$. As it has been observed, the multi-linear component of $\mathsf{SSI}_Q(\alpha, b)$ is isomorphic to
\[(\otimes_{c\in Q_0} ((\Pi^{b(c)}V(c)^*)^{\otimes\mathsf{in}(c)}\otimes \Pi^{b(c)}V(c)^{\otimes\mathsf{out}(c)}))^{\SL(\alpha)}.\]
Let $\beta$ denote the super-dimension vector $(\alpha(b))_{b\in Q_0\setminus a}$. Let $U$ and $U'$ denote the $\SL(\alpha)$-supermodule
\[\otimes_{c\in Q_0} ((\Pi^{b(c)}V(c)^*)^{\otimes\mathsf{in}(c)}\otimes \Pi^{b(c)}V(c)^{\otimes\mathsf{out}(c)})\]
and the $\SL(\beta)$-supermodule
\[\otimes_{c\in Q_0\setminus a} ((\Pi^{b(c)}V(c)^*)^{\otimes\mathsf{in}(c)}\otimes \Pi^{b(c)}V(c)^{\otimes\mathsf{out}(c)})\]
respectively.

Note that $\SL(\alpha(a))=\SL(\alpha(a)_1)$ is an ordinary special linear group. Moreover, $\Pi^{b(a)}V(a)$ (as well as its dual) can be regarded as an ordinary $\SL(\alpha(a)_1)$-module, no matter if it is purely odd as a superspace.

The tensor factor
\[W=(\Pi^{b(a)}V(a)^*)^{\otimes\mathsf{in}(a)}\otimes \Pi^{b(a)}V(a)^{\otimes\mathsf{out}(a)}\]
can be regarded as a $\SL(\alpha(a))\times\SL(\alpha(a))$-module, where the first factor $\SL(\alpha(a))$ acts diagonally on $(\Pi^{b(a)}V(a)^*)^{\otimes\mathsf{in}(a)}$ and the second factor
$\SL(\alpha(a))$ acts diagonally on $\Pi^{b(a)}V(a)^{\otimes\mathsf{out}(a)}$.
Using Proposition 12, \cite{zubshest} (see also \cite{jan}, I.3.4, I.3.6, or \cite{gross}, Theorem 9.1), one obtains
\[W^{\SL(\alpha(a))}\simeq (W\otimes
K[(\SL(\alpha(a))\times\SL(\alpha(a)))/\SL(\alpha(a))])^{\SL(\alpha(a))\times\SL(\alpha(a))},\]
where $\SL(\alpha(a))$ is diagonally embedded into $\SL(\alpha(a))\times\SL(\alpha(a))$. Besides, this isomorphism is just identity map on $W$ and takes
\[
f\in K[(\SL(\alpha(a))\times\SL(\alpha(a)))/\SL(\alpha(a))]
\]
to $f((E\times E)\SL(\alpha(a)))$.

The quotient
\[(\SL(\alpha(a))\times\SL(\alpha(a)))/\SL(\alpha(a))\]
is isomorphic to $\SL(\alpha(a))$ via
\[
(x, y)\SL(\alpha(a))\mapsto xy^{-1}, \qquad (x, y)\in\SL(\alpha(a))\times\SL(\alpha(a)).\]
Moreover, the group $\SL(\alpha(a))\times\SL(\alpha(a))$ acts on the latter scheme as
\[
(g, h)\cdot x=gxh^{-1}, \qquad (g, h)\in\SL(\alpha(a))\times\SL(\alpha(a)), \quad x\in\SL(\alpha(a)).\]
Therefore, there is a natural isomorphism of $\SL(\alpha(a))\times\SL(\alpha(a))$-modules
\[K[(\SL(\alpha(a))\times\SL(\alpha(a)))/\SL(\alpha(a))]\simeq A/A(\det(X(e(a)))-1),\]
where $A=K[x_{ij}(e(a))\mid 1\leq i, j\leq\alpha(a)_1]$.

Since the category of rational $\SL(\alpha(a))\times\SL(\alpha(a))$-modules is semisimple, we have an epimorphism
\[(W\otimes A)^{\SL(\alpha(a))\times\SL(\alpha(a))}\to (W\otimes A/A(\det(X(e(a)))-1))^{\SL(\alpha(a))\times\SL(\alpha(a))},\]
which induces an epimorphism $(W\otimes A)^{\SL(\alpha(a))\times\SL(\alpha(a))}\to W^{\SL(\alpha(a))}$. Moreover, the latter epimorphism is an identity map on $W$ and takes $X(e(a))$ to $E$.

Multiplying by $U'^{\SL(\beta)}$, we obtain an epimorphism
$(U\otimes A)^{\SL(\alpha')}\to U^{\SL(\alpha)}$, which is an identity map on $U$ and takes $X(e(a))$ to $E$ as well. It remains to note that $(U\otimes A)^{\SL(\alpha')}$ is nothing but the subspace of $\mathsf{SSI}_{Q'}(\alpha', b')$ consisting of all semi-invariants, those have degree one in each $X(e), e\in Q_1$.
Proposition is proven.
\end{proof}
{Applying the above proposition to extremal vertices $a\in Q_0$, one can construct a quiver $\widetilde{Q}$ with a super-dimension vector $\widetilde{\alpha}$ and a parity vector $\widetilde{b}$, such that each extremal vertex of $\widetilde{Q}$ is either a source or a sink. Moreover, the superalgebra morphism
$\mathsf{SSI}_{\widetilde{Q}}(\widetilde{\alpha}, \widetilde{b})\to\mathsf{SSI}_Q(\alpha, b)$, induced by the identity map on $\Pi^{\widetilde{b}}X_{\widetilde{\alpha}}(e)=\Pi^b X_{\alpha}(e)$ for $e\in Q_1$, and which takes $\Pi^{\widetilde{b}} X_{\widetilde{\alpha}}(e)$ for $e\in\widetilde{Q}_1\setminus Q_1$ to the identity matrix of the same size, is surjective.}

{If $f\in \mathsf{SSI}_Q(\alpha, b)$ is the image of a determinant-like invariant of $\widetilde{Q}$, then it is also called determinant-like.}

\subsection{Computing of semi-invariants}

{According to the above remark, we can assume that each extremal vertex of $Q$ is either a source or a sink.} Denote $v_i^*(h(e))$  by $v_i^*|e|$, and  $v_j(t(e))$ by $v_j|e|$.

First of all, assume  $a$ is an ordinary vertex. Theorem 1.2 of \cite{sergeev} states that
\[
\Big((\Pi^{b(a)}V(a)^*)^{\otimes\mathsf{in}(a)}\otimes\Pi^{b(a)}V(a)^{\otimes\mathsf{out}(a)}\Big)^{\mathfrak{sl}(\alpha(a))}\neq 0\]
if and only if $\mathsf{in}(a)=\mathsf{out}(a)$. This equality is called the {\it  Kirchhoff rule}. If $a$ satisfies this rule, then $((\Pi^{b(a)}V(a)^*)^{\otimes\mathsf{in}(a)}\otimes\Pi^{b(a)}V(a)^{\otimes\mathsf{out}(a)})^{\mathfrak{sl}(\alpha(a))}$ is generated (as a vector space) by the products
\[
\otimes_{e\in\mathsf{OUT}(a)}\Big(\sum_{1\leq i\leq \alpha(a)_0+\alpha(a)_1}\Pi^{b(a)}v_i|e|\otimes \Pi^{b(a)}v^*_i|f_a(e)|\Big),
\]
where $f_a : \mathsf{OUT}(a)\to\mathsf{IN}(a)$ is a bijection. The last products can be considered as elements of
$((\Pi^{b(a)}V(a)^*)^{\otimes\mathsf{in}(a)}\otimes\Pi^{b(a)}V(a)^{\otimes\mathsf{out}(a)})^{\mathfrak{sl}(\alpha(a))}$ after a permutation of its tensor factors, taking into account the sign rule.

Assume that $a$ is an extremal vertex that is a source. By Theorem 1.2 of \cite{sergeev}, the tensor factor $(\Pi^{b(a)}V(a)^{\otimes\mathsf{out}(a)})^{\mathfrak{sl}(\alpha(a))}$ is non-zero if and only if $d_{\alpha}(a)$ divides $\mathsf{out}(a)$ and $\Pi^{b(a)}V(a)$ is a purely even superspace.  Moreover, if the latter holds, then the generators of the vector space $(\Pi^{b(a)}V(a)^{\otimes\mathsf{out}(a)})^{\mathfrak{sl}(\alpha(a))}$ can be described as follows.

Denote by $T$ a partition of $\mathsf{OUT}(a)$ into a union of disjoint subsets, each of which has a cardinality
$d_{\alpha}(a)$, say $\mathsf{OUT}(a)=\sqcup_{1\leq j\leq k_{\alpha}(a)} T_j$, where $d_{\alpha}(a)k_{\alpha}(a)=\mathsf{out}(a)$. Then the elements
\[
\otimes_{1\leq j\leq k_{\alpha}(a)}\det\Big((\Pi^{b(a)}v_i|e|)_{1\leq i\leq d_{\alpha}(a), e\in T_j}\Big),\]
corresponding to various partitions $T$ as above,
after reordering of the tensor factors,
generate $(\Pi^{b(a)}V(a)^{\otimes\mathsf{out}(a)})^{\mathfrak{sl}(\alpha(a))}$ as a vector space.

Symmetrically, if $a$ is an extremal sink, then $((\Pi^{b(a)}V(a)^*)^{\otimes\mathsf{in}(a)})^{\mathfrak{sl}(\alpha(a))}\neq 0$ if and only if $d_{\alpha}(a)$ divides $\mathsf{in}(a)$ and $\Pi^{b(a)}V(a)^*$ is a purely even superspace. The generators of the vector space $((\Pi^{b(a)}V(a)^*)^{\otimes\mathsf{in}(a)})^{\mathfrak{sl}(\alpha(a))}$ have the form
\[\otimes_{1\leq j\leq k_{\alpha}(a)}\det((\Pi^{b(a)}v^*_i|e|)_{1\leq i\leq d_{\alpha}(a), e\in T_j}),\]
after reordering of the tensor factors, where each $T_j$ has the cardinality $d_{\alpha}(a)$,
 $d_{\alpha}(a)k_{\alpha}(a)=\mathsf{in}(a)$ and $\sqcup_{1\leq j\leq k_{\alpha}(a)} T_j=\mathsf{IN}(a)$.

Without a loss of generality, one can assume that the quiver $Q$ does not have isolated vertices.
Then the multi-linear component of $\mathsf{SSI}_Q(\alpha, b)$ is generated by the products of the above elements, take one for each vertex. Let us call such products {\it vertex decomposable} semi-invariants.

It is clear that every source and sink is an extremal vertex. On the other hand, any ordinary vertex is neither source nor sink.

Let $\mathsf{IN}(Q)$ denote $\sqcup_{a\in\mathsf{SINK(Q)}}\mathsf{IN}(a)$, and $\mathsf{OUT}(Q)$ denote
$\sqcup_{a\in\mathsf{SOURCE}(Q)}\mathsf{OUT}(a)$.
\begin{theorem}\label{polynomial semi-invariants}
The superalgebra $\mathsf{SSI}_Q(\alpha, b)$ is generated by the invariants $\mathrm{str}(\Pi^b X(p))$, where $p$ runs over all closed paths in $Q$, and by all determinant-like semi-invariants.
\end{theorem}
\begin{proof}
{As above, one can assume that each extremal vertex of $Q$ is either a source or a sink.} Let us consider a polynomial vertex decomposable multi-linear semi-invariant $F$ of the total degree $t>0$.

If there are no ordinary vertices, then the statement follows from Theorem 4.1 of \cite{domzub}. Otherwise, consider an ordinary vertex $a$.

If there is $e_1\in\mathsf{OUT}(a)$ such that $a'=t(f_a(e_1))$ is ordinary, then we look for $e_2\in\mathsf{OUT}(a')$ such that $a''=t(f_{a'}(e_1))$ is ordinary. We continue like this, repeating as many times as needed until we terminate either at the initial vertex $a$ or at a source. In the first case, we obtains a closed path $p$, consisting of edges $e_k, e_{k-1, }\ldots, e_1$, such that :
\begin{enumerate}
\item all vertices $a=a_1=t(e_1), \ldots, a_k=t(e_k)$ are ordinary;
\item $f_{a_k}(e_k)=e_1$ and $f_{a_i}(e_i)=e_{i+1}$,  $1\leq i\leq k-1$.
\end{enumerate}
Any path like that is called an {\it $F$-related} path.

Let us consider a subproduct
\[P=\otimes_{1\leq j\leq k}(\sum_{1\leq i\leq d_{\alpha}(a_j)}\Pi^{b(a_j)}v_i|e_j|\otimes \Pi^{b(a_j)}v_i^*|f_{a_j}(e_j)|)\]
of $F$, i.e., $F=PF'$, where $F'$ is a vertex decomposable semi-invariant of the total degree strictly less than $t$.

For the sake of simplicity, denote $\Pi^{b(h(e))}v_i^*|e|$ by $w_i^*|e|$, and $\Pi^{b(t(e))}v_j|e|$ by $w_j|e|$. Also, denote $d_i=d_{\alpha}(a_i)$ for $1\leq i\leq k$. Then
\[
\begin{aligned}
P=&\sum_{1\leq i_s\leq d_s, 1\leq s\leq k} w_{i_1}|e_1|\otimes w^*_{i_1}|e_2|\otimes\cdots\otimes w_{i_k}|e_k|\otimes w^*_{i_k}|e_1 |\\
=&\sum_{1\leq i_s\leq d_s, 1\leq s\leq k}(-1)^{|w_{i_k}||e_k|}\Pi^v x_{i_k, i_1}(e_1)\Pi^b x_{i_1, i_2}(e_2)\cdots\Pi^b x_{i_{k-1}, i_k}(e_k)\\
=&\, \mathrm{str}\big(\Pi^b X(e_1)\cdots\Pi^b X(e_k)\big)=\mathrm{str}\big(\Pi ^b X(p)\big).
\end{aligned}
\]
By induction on $t$, we can assume that for any ordinary vertex $a$ the above procedure terminates at a source. In other words, there are no $F$-related paths. In this case, $F$ is called a {\it bipartite} semi-invariant.

Now assume that $F$ is a bipartite semi-invariant and $a$ is a sink.

The Kirchhoff rule implies that, for each $e\in\mathsf{IN}(a)$, there is a unique path $p_e$ consisting of edges $e_{k(e)}, \ldots , e_1=e,$ such that $a_{k(e)}=t(e_{k(e)})$ is a source, all vertices $a_i=t(e_i)$ for $1\leq i\leq k(e)-1$ are ordinary, and $f_{a_i}(e_i)=e_{i+1}$. Conversely, for any source $c$ and an arbitrary edge $e'\in\mathsf{OUT}(c)$, there is a unique sink $a$ and a unique edge $e\in\mathsf{IN}(a)$ such that $c$ is the initial vertex and $e'$ is the first edge of $p_e$. In particular, the sets $\mathsf{IN}(Q)$and $\mathsf{OUT}(Q)$ have the same cardinality. Every edge of $Q$ appears as an edge in some $p_e$. Moreover, different elements $p_e$ have no common edge.

To describe the structure of $F$, as a vertex decomposable semi-invariant, we need some additional notations.

First of all, for every $a\in\mathsf{SINK}(Q)$ there is a partition $\{T_j(a)\}_{1\leq j\leq k_{\alpha}(a)}$ of the subset $\mathsf{IN}(a)$, and for every $c\in\mathsf{SOURCE}(Q)$ there is a partition $\{T'_j(c)\}_{1\leq j\leq k_{\alpha}(b)}$ of the subset $\mathsf{OUT}(c)$, uniquely defined by $F$.
Choose a map 
\[
\mathsf{IN}(Q)\to\sqcup_{a\in\mathsf{SINK}(Q)}[1, d_{\alpha}(a)]
\] 
such that its restrictions on each $T_j(a)$ is a bijection onto $[1, d_{\alpha}(a)]$, and a map
$\mathsf{OUT}(Q)\to\sqcup_{c\in\mathsf{SOURCE}(Q)}[1, d_{\alpha}(c)]$ such that its restriction on each $T'_j(c)$ is a bijection onto $[1, d_{\alpha}(c)]$.
The image of $e\in\mathsf{IN}(Q)\sqcup\mathsf{OUT}(Q)$ under one of these maps is denoted by $\bar e$. Then
\[
\begin{aligned}
F=&\sum_{\sigma\in S_T, \tau\in S_{T'}}(-1)^{sign(\sigma)+sign(\tau)}\prod_{e\in\mathsf{IN}(Q)}\sum_{i_1, \ldots, i_{k(e)}} w^*_{\sigma(\overline{e_1})}|e_1|\otimes w_{i_1}|e_1|\otimes \\
&w^*_{i_1}|e_2|\otimes \cdots\otimes w_{i_{k(e)}}|e_{k(e)-1}|\otimes w^*_{i_{k(e)}}|e_{k(e)}|\otimes w_{\tau(\overline{e_{k(e)}})}|e_{k(e)}|,
\end{aligned}
\]
where $S_T=\times_{a\in\mathsf{SINK}(Q)} S_{d_{\alpha}(a)}^{k_{\alpha}(a)}$ and $S_{T'}=\times_{a\in\mathsf{SOURCE}(Q)} S_{d_{\alpha}(a)}^{k_{\alpha}(a)}$.
A substitution
\[
\sigma=\times_{a\in\mathsf{SINK}(Q)}(\times_{1\leq j\leq k_{\alpha}(a)}\sigma_j(a))\in S_T
\]
acts on $e_1\in T_j(h(e_1))$ as $w^*_{\sigma(\overline{e_1})}|e_1|=w^*_{\sigma_j(h(e_1))(\overline{e_1})}|e_1|$. Analogously, a substitution
\[\tau=\times_{c\in\mathsf{SOURCE}(Q)}(\times_{1\leq j\leq k_{\alpha}(c)}\tau_j(a))\in S_{T'}\]
acts on $e_{k(e)}\in T_j(t(e_{k(e)}))$ as $w_{\tau(\overline{e_{k(e)}})}|e_{k(e)}|=w_{\tau_j(t(e_{k(e)}))(\overline{e_{k(e)}})}|e_{k(e)}|$.

Finally, for every $e\in\mathsf{IN}(Q)$ and $1\leq s\leq k$, the index $i_s$ varies within the interval $[1, d_{\alpha}(h(e_s))]$.

Further, we can rewrite $F$ as
\[
\begin{aligned}
F=&\sum_{\sigma\in S_T, \tau\in S_{T'}}(-1)^{sign(\sigma)+sign(\tau)}\prod_{e\in\mathsf{IN}(Q)}\sum_{i_1, i_2, \ldots, i_{k(e)}} \Pi^b x_{\sigma(\overline{e_1}), i_1}(e_1)\cdots \Pi^b x_{i_{k(e)}, \tau(\overline{e_{k(e)}})}\\
=&\sum_{\sigma\in S_T, \tau\in S_{T'}}(-1)^{sign(\sigma)+sign(\tau)}\prod_{e\in\mathsf{IN}(Q)}\Pi^b x_{\sigma(\overline{e_1}), \tau(\overline{e_{k(e)}})}(p_e).
\end{aligned}
\]
In other words, $F$ can be interpreted as a multi-linear bipartite semi-invariant of a
bipartite quiver $Q'$ such that
\[
Q'_0=\mathsf{SINK}(Q)\sqcup\mathsf{SOURCE}(Q)\quad  \text{ and }\quad  Q'_1=\{p_e\mid e\in\mathsf{IN}(Q)\},
\]
where the functions $t$ and $h$ are defined as $t(p_e)=t(e_{k(e)})$ and $h(p_e)=h(e_1)$, and
$p_e$ consists of the edges $e_{k(e)}, \ldots, e_1$. Theorem 4.1 from \cite{domzub} concludes the proof.
\end{proof}
\begin{corollary}\label{strangephenomenon}
If $Q$ has no extremal vertices with respect to the super-dimension vector $\alpha$, then
$\mathsf{SSI}_Q(\alpha, b)=\mathsf{SI}_Q(\alpha, b)$. Moreover, if $Q$ is acyclic, then $\mathsf{SSI}_Q(\alpha, b)=K$.
\end{corollary}
This corollary shows that the polynomial semi-invariants of super-representations of a quiver reflect
its {\it geometrical} properties, but in a drastically different way than the polynomial semi-invariants of its usual representations.

This strange phenomenon does not take place if we consider {\it rational} semi-invariants of super-representations of a quiver. Regardless whether all vertices are ordinary or not, one can construct a non-scalar rational semi-invariant as in the following example.
\begin{ex}\label{phenomena}
Let $Q$ be the Kronecker quiver with the vertex set $Q_0=\{a, b\}$, the edge set $Q_1=\{e_1, e_2\}$ and $h(e_i)=b, t(e_i)=a$ for $i\in\{0, 1\}$.

Assume the super-dimension vector $\alpha=(p|0, q|0)$. Then $\mathsf{SRep}_{\alpha}(Q)=\mathsf{Rep}_{\alpha}(Q)$. Both generic matrices $X(e_1)$ and $X(e_2)$ are of the size
$q\times p$. Choose a pair of positive integers $s, l$ such that $qs=pl=t$ and construct a matrix $X$ of the size $t\times t$, partitioned into $s\times l$ blocks of size $q\times p$. The block at the position $(i, j)$, where
$1\leq i\leq s, 1\leq j\leq l$, is defined as $\beta_{ij}X(e_1)+\gamma_{ij}X(e_2)$, where the coefficients $\beta_{ij}, \gamma_{ij}$ are algebraically independent over the fraction field of $K[\mathsf{Rep}_{\alpha}(Q)]$. Then $\det(X)\neq 0$ if and only if there are values of the variables  $\beta_{ij}, \gamma_{ij}$ in the field $K$, say $b_{ij}, g_{ij}$,  for which
\[\det(X)|_{\beta_{ij}\mapsto b_{ij}, \gamma_{ij}\mapsto g_{ij}, 1\leq i\leq s, 1\leq j\leq l}\]
is a non-zero semi-invariant of $Q$.

Assume the super-dimension vector $\alpha=(m|n, p|q)$ is such that $mn\neq 0$ and  $pq\neq 0$. Then both vertices of the quiver are ordinary for such $\alpha$, and $\mathsf{SSI}_Q(\alpha)=K$.
Additionally, assume that there is a pair of positive integers $s, l$ such that $sm|sn=lp|lq$. We construct a matrix $X$ of the size $t\times t$,
partitioned into $s\times l$ blocks of "super-size" $p|q\times m|n$. We define the block at the position $(i, j)$, where $1\leq i\leq s$, $1\leq j\leq l$, as $\beta_{ij}X(e_1)+\gamma_{ij}X(e_2)$, where all coefficients $\beta_{ij}, \gamma_{ij}$ are algebraically independent over the fraction field of
$K[\mathsf{Rep}_{\alpha}(Q)]_0$. Up to some rearrangement of its rows and columns, $X$ is an ${sm|sn\times sm|sn}$ super-matrix. Moreover, $\Ber_{sm|sn}(X)$ is defined correctly, i.e., $X$ is invertible if and only if for some values of variables  $\beta_{ij}, \gamma_{ij}$ in the field $K$, say $b_{ij}, g_{ij}$,  the element
\[\Ber_{sm|sn}(X)|_{\beta_{ij}\mapsto b_{ij}, \gamma_{ij}\mapsto g_{ij}, 1\leq i\leq s, 1\leq j\leq l}\]
is invertible. In particular, the latter element is a non-zero rational semi-invariant of $Q$.

\end{ex}

Finally, arguing as in Theorem \ref{polynomial semi-invariants} and using Theorem 1.1 from \cite{sergeev}, one can easily derive the following statement (see also \cite{brprocesi, don}).
\begin{corollary}\label{polynomialinvariants}
The superalgebra $\mathsf{SI}_Q(\alpha, b)$ is generated by the invariants $\mathrm{str}(\Pi^b X(p))$, where $p$ is a closed path in $Q$.
\end{corollary}



\end{document}